\documentclass{amsart}
\usepackage{amsmath}
\usepackage{amssymb}
\usepackage{enumerate}
\usepackage{enumitem}
\usepackage{mathrsfs}
\usepackage{cases}
\usepackage{color}
\usepackage{tikz}
\usepackage{appendix}
\usepackage{makecell}
\usepackage{multirow}
\usepackage{algpseudocode}
\usepackage{algorithmicx,algorithm}
\usepackage{graphicx}
\usepackage{mathtools}
\definecolor{darkblue}{rgb}{0.0, 0.0, 0.55}
\definecolor{bordeaux}{rgb}{0.34, 0.01, 0.1}
\usepackage[colorlinks,linkcolor=bordeaux,citecolor=darkblue,urlcolor=black,hypertexnames=true]{hyperref}
\usetikzlibrary{arrows,positioning}

\newtheorem{theorem}{Theorem}[section]

\newtheorem{example}[theorem]{Example}

\newtheorem{proposition}[theorem]{Proposition}

\newtheorem{remark}[theorem]{Remark}
\numberwithin{equation}{section}
\def\Z{{\mathbb{Z}}}

\def\R{{\mathbb{R}}}
\def\C{{\mathbb{C}}}
\def\N{{\mathbb{N}}}

\def\M{{\mathbf{M}}}
\def\x{{\mathbf{x}}}
\def\z{{\mathbf{z}}}

\def\y{{\mathbf{y}}}

\def\o{{\mathbf{o}}}

\def\a{{\boldsymbol{\alpha}}}

\def\b{{\boldsymbol{\beta}}}
\def\g{{\boldsymbol{\gamma}}}

\newcommand{\br}{\mathbf{r}}

\def\A{{\mathscr{A}}}
\def\B{{\mathscr{B}}}
\def\CC{{\mathscr{C}}}

\def\H{{\mathbf{H}}}

\def\i{\hbox{\bf{i}}}
\def\supp{\hbox{\rm{supp}}}

\def\int{\hbox{\rm{int}}}
\def\st{\hbox{\rm{s.t.}}}

\newcommand{\vge}{\mathbin{\rotatebox[origin=c]{90}{$\ge$}}}
\newif\ifcomment
\commentfalse
\commenttrue
\usepackage{todonotes}

\newcommand{\revision}[1]{{{\color{black}#1}}}
\begin{document}


\title[Exploiting Sparsity in Complex Polynomial Optimization]{Exploiting Sparsity in Complex Polynomial Optimization}

\author[J. Wang]{Jie Wang}
\address{Academy of Mathematics and Systems Science, CAS}
\email{wangjie212@amss.ac.cn}
\urladdr{https://wangjie212.github.io/jiewang}

\author[V. Magron]{Victor Magron}
\address{Laboratory for Analysis and Architecture of Systems, CNRS}
\email{vmagron@laas.fr}
\urladdr{https://homepages.laas.fr/vmagron}


\date{\today}

\begin{abstract}
In this paper, we study the sparsity-adapted complex moment-Hermitian sum of squares (moment-HSOS)  hierarchy for complex polynomial optimization problems, where the {\em sparsity} includes correlative sparsity and term sparsity. We compare the strengths of the sparsity-adapted complex moment-HSOS hierarchy with the sparsity-adapted real moment-SOS hierarchy on either randomly generated complex polynomial optimization problems or the AC optimal power flow problem. The results of numerical experiments show that the sparsity-adapted complex moment-HSOS hierarchy provides a trade-off between the computational cost and the quality of obtained bounds for large-scale complex polynomial optimization problems.
\newline
\newline
$^\star$ Communicated by Vaithilingam Jeyakumar
\end{abstract}

\maketitle

\section{Introduction}
In this paper, we consider the following complex polynomial optimization problem (CPOP):
\begin{equation}\label{cpop}
(\textrm{Q}):\quad
\begin{cases}
\inf_{\z\in\C^n} &f(\z,\bar{\z})\coloneqq\sum_{\a,\b}f_{\a,\b}\z^{\a}\bar{\z}^{\b}\\
\textrm{s.t.}&g_j(\z,\bar{\z})\coloneqq\sum_{\a,\b}g_{j,\a,\b}\z^{\a}\bar{\z}^{\b}\ge0,\quad j=1,\ldots,m,\\
&h_i(\z,\bar{\z})\coloneqq\sum_{\a,\b}h_{i,\a,\b}\z^{\a}\bar{\z}^{\b}=0,\quad i=1,\ldots,t,
\end{cases}
\end{equation}
where $n$, $m$, and $t$ are positive integers,  $\bar{\z}\coloneqq(\bar{z}_1,\ldots,\bar{z}_n)$ stands for the conjugate of complex variables $\z\coloneqq(z_1,\ldots,z_n)$. The functions $f,g_1,\ldots,g_m,h_1,\ldots,h_t$ are real-valued polynomials and their coefficients satisfy $f_{\a,\b}=\bar{f}_{\b,\a}$, $g_{j,\a,\b}= \bar{g}_{j,\b,\a}$, and $h_{i,\a,\b}= \bar{h}_{i,\b,\a}$. The feasible set is defined as $\{\z\in\C^n\mid g_j(\z, \bar{\z}) \ge 0, j=1,\ldots,m,h_i(\z, \bar{\z}) = 0, i=1,\ldots,t\}$. For the sake of brevity, we assume that there are only inequality constraints in \eqref{cpop} for the rest of this paper.
CPOP \eqref{cpop} arises naturally from diverse areas, such as imaging science \cite{fogel2016phase}, signal processing \cite{aittomaki2009,aubry2013ambiguity,mariere2003}, automatic control \cite{toker1998complexity}, quantum mechanics \cite{hilling2010}, optimal power flow \cite{bienstock2020}. 
By introducing real variables for the real part and the imaginary part of each complex variables respectively, CPOP \eqref{cpop} can be converted into a polynomial optimization problem (POP) involving only real variables.

\revision{The moment-sum of squares (SOS) hierarchy (also known as Lasserre's hierarchy) \cite{Las01}, which consists of a sequence of increasingly tight semidefinite relaxations, has become a popular tool to retrieve global optimal values of POPs involving real variables. For convenience, we refer to the moment-SOS hierarchy as the ``real hierarchy'' in this paper. The real hierarchy can be further adapted as the moment-Hermitian sum of squares (HSOS) hierarchy to handle CPOPs \cite{josz2018lasserre}, which thereby we refer to as the ``complex hierarchy'' in this paper. Notice that for a CPOP, one can either apply the real hierarchy after converting it into a real POP or apply the complex hierarchy directly. It is known that the complex hierarchy never produces tighter bounds than the real hierarchy when using the same relaxation order, which is, however, still of interest because of its lower computational complexity as recently shown in \cite{josz2018lasserre}.}

On the other hand, due to the rapidly growing size of semidefinite relaxations, the standard real and complex hierarchies are typically solvable only for modest-size problems on a personal computer.
To improve their scalability, it is then crucial to exploit the structure, e.g., sparsity, encoded in the problem data. For the real hierarchy, one can make use of the well-known correlative sparsity \cite{waki}, or the newly proposed term sparsity \cite{tssos2,tssos1}, or the combination of both \cite{tssos3}. 
Exploiting sparsity in the real hierarchy has been successfully done for many practical applications, including computer arithmetic \cite{toms17,toms18}, control \cite{vreman2021stability,wang2020sparsejsr}, machine learning \cite{chen2020polynomial}, noncommutative optimization \cite{klep2021sparse,nctssos,zhou2020proper,zhou2020fairness}, rational function optimization \cite{bugarin2016minimizing}, just to name a few. 
For the complex hierarchy, one can also make use of correlative sparsity; see \cite{josz2018lasserre} for an application to the AC optimal power flow problem. The main purpose of this paper is to develop sparsity-adapted complex hierarchies by taking into account term sparsity as for the real hierarchy.

{\bf Contribution.} Our contributions are threefold:
\begin{enumerate}
    \item[1)] We propose sparsity-adapted complex hierarchies based on either term sparsity or correlative-term sparsity for CPOPs, which are indexed by two parameters: the relaxation order $d$ and the sparse order $k$. The optima of the sparsity-adapted complex hierarchies for a fixed $d$ are proved to converge to the optimum of the dense relaxation or the relaxation exploiting only correlative sparsity with the same relaxation order when the maximal chordal extension is chosen. We also prove that the block structure arising in the sparsity-adapted complex hierarchies is always a refinement of the block structure determined by the sign symmetries of the problem.
    \item[2)] We propose a minimum initial relaxation step of the sparsity-adapted complex hierarchy. For the AC optimal power flow problem, this new relaxation is able to provide a tighter lower bound than Shor's relaxation and is, \revision{depending on the input}, possibly much less expensive than the second order relaxation of the complex hierarchy.
    \item[3)] We provide a comprehensive comparison on the strengths of the sparsity-adapted real hierarchy and the sparsity-adapted complex hierarchy for CPOPs via numerical experiments. The largest numerical example is an instance of the AC optimal power flow problem involving $2869$ complex variables (or $5738$ real variables).
\end{enumerate}

Our sparsity-adapted complex hierarchies can be viewed as complex variants of the ones obtained for the real case \cite{tssos2,tssos1,tssos3}. 
Throughout the paper, we emphasize in several places the subtle differences between the complex and real settings, in particular to define sparsity-adapted moment/localizing matrices and the connection with sign symmetries when (finite) convergence occurs. 
We also hope that it is of interest for researchers solving large-scale CPOPs (e.g., AC optimal power flow problems) to have a self-contained paper explaining in detail the construction of term sparsity pattern graphs, as well as the sparsity-adapted semidefinite formulations.
Last but not least, we do not pretend that the sparsity-adapted complex hierarchy systematically provides better results than the real hierarchy.
It rather provides a trade-off between efficiency and accuracy for large-scale CPOPs.

The rest of the paper is organized as follows. In Section \ref{preliminaries}, we recall some preliminary background. In Section \ref{hierarchies}, we establish the sparsity-adapted complex hierarchies and prove some of their properties. In Section \ref{initial}, a minimum initial relaxation step for the sparsity-adapted complex hierarchy is introduced.
Numerical experiments are provided in Section \ref{experiments} and conclusions are provided in Section \ref{cons}.

\section{Notation and preliminaries}\label{preliminaries}
Let $\N$ be the set of nonnegative integers.
For $n\in\N\backslash\{0\}$, let $[n]\coloneqq\{1,2,\ldots,n\}$.
Let $\i$ be the imaginary unit, satisfying $\i^2 = -1$.
Let $\z=(z_1,\ldots,z_n)$ (resp. $\x=(x_1,\ldots,x_n)$) be a tuple of complex (resp. real) variables and $\bar{\z}=(\bar{z}_1,\ldots,\bar{z}_n)$ its conjugate.
We denote by $\C[\z]\coloneqq\C[z_1,\ldots,z_n]$, $\C[\z,\bar{\z}]\coloneqq\C[z_1,\ldots,z_n,\bar{z}_1,\ldots,\bar{z}_n]$, $\R[\x]\coloneqq\R[x_1,\ldots,x_n]$ the complex polynomial ring in $\z$, the complex polynomial ring in $\z,\bar{\z}$, the real polynomial ring in $\x$, respectively. For $d\in\N$, let $\C_{d}[\z]$ (resp. $\C_{d}[\z,\bar{\z}]$) denote the set of polynomials in $\C[\z]$ (resp. $\C[\z,\bar{\z}]$) of degree no greater than $d$. A polynomial $f\in\C[\z,\bar{\z}]$ can be written as $f=\sum_{(\b,\g)\in\A}f_{\b,\g}\z^{\b}\bar{\z}^{\g}$ with $\A\subseteq\N^n\times\N^n$ and $f_{\b,\g}\in\C, \z^{\b}=z_1^{\beta_1}\cdots z_n^{\beta_n},\bar{\z}^{\g}=\bar{z}_1^{\gamma_1}\cdots \bar{z}_n^{\gamma_n}$. The support of $f$ is defined by $\supp(f)=\{(\b,\g)\in\A\mid f_{\b,\g}\ne0\}$. The conjugate of $f$ is $\bar{f}=\sum_{(\b,\g)\in\A}\bar{f}_{\b,\g}\z^{\g}\bar{\z}^{\b}$. A polynomial $\sigma=\sum_{(\b,\g)}\sigma_{\b,\g}\z^{\b}\bar{\z}^{\g}\in\C_{2d}[\z,\bar{\z}]$ is called a {\em Hermitian sum of squares} or an {\em HSOS} for short if there exist polynomials $f_i\in\C_{d}[\z], i\in[t]$ such that $\sigma=\sum_{i=1}^tf_i\bar{f}_i$. We will use $|\cdot|$ to denote the cardinality of a set. For $(\b,\g)\in\N^n\times\N^n$, $\A\subseteq\N^n\times\N^n$, let $(\b,\g)+\A\coloneqq\{(\b+\b',\g+\g')\mid(\b',\g')\in\A\}$.

For a positive integer $r$, the set of $r\times r$ Hermitian matrices is denoted by $\H^r$ and the set of $r\times r$ positive semidefinite (PSD) Hermitian matrices is denoted by $\H_+^r$. Let $A\circ B\in\H^r$ denote the Hadamard product of $A,B\in\H^r$, defined by $[A\circ B]_{ij} = A_{ij}B_{ij}$.
For $d\in\N$, let $\N^n_d\coloneqq\{(\alpha_i)_{i}\in\N^n\mid\sum_{i=1}^n\alpha_i\le d\}$ (sorted with respect to the lexicographic order). The set $\{\z^{\b}\mid\b\in\N^n_d\}$ is called the standard (complex) {\em monomial basis} up to degree $d$. For the sake of convenience, we abuse notation slightly in this paper and use the exponent set $\N^n_{d}$ to denote the monomial basis. 

\subsection{The complex moment-HSOS hierarchy}\label{complexsos}
Let $\y=(y_{\a})_{\a\in\N^n}\in\R^{\N^n}$ be a sequence indexed by $\a\in\N^n$. Let $L^{\textrm{r}}_{\y}:\R[\x]\rightarrow\R$ be the linear functional
\begin{equation*}
f=\sum_{\a}f_{\a}\x^{\a}\mapsto L^{\textrm{r}}_{\y}(f)=\sum_{\a}f_{\a}y_{\a}.
\end{equation*}
The {\em real moment} matrix $\M^{\textrm{r}}_{d}(\y)$ ($d\in\N$) associated with $\y$ is the matrix with rows and columns indexed by $\N^n_{d}$ such that
\begin{equation*}
\M^{\textrm{r}}_d(\y)_{\b\g}\coloneqq L^{\textrm{r}}_{\y}(\x^{\b}\x^{\g})=y_{\b+\g}, \quad\forall\b,\g\in\N^n_{d}.
\end{equation*}
Suppose $g=\sum_{\a}g_{\a}\x^{\a}\in\R[\x]$. The {\em real localizing} matrix $\M^{\textrm{r}}_{d}(g\y)$ associated with $g$ and $\y$ is the matrix with rows and columns indexed by $\N^n_{d}$ such that
\begin{equation*}
\M^{\textrm{r}}_{d}(g\,\y)_{\b\g}\coloneqq L^{\textrm{r}}_{\y}(g\,\x^{\b}\x^{\g})=\sum_{\a}g_{\a}y_{\a+\b+\g}, \quad\forall\b,\g\in\N^n_{d}.
\end{equation*}

Now let $\y=(y_{\b,\g})_{(\b,\g)\in\N^n\times\N^n}\in\C^{\N^n\times\N^n}$ be a sequence indexed by $(\b,\g)\in\N^n\times\N^n$ and satisfies $y_{\b,\g}=\bar{y}_{\g,\b}$. Let $L^{\textrm{c}}_{\y}:\C[\z,\bar{\z}]\rightarrow\R$ be the linear functional
\begin{equation*}
f=\sum_{(\b,\g)}f_{\b,\g}\z^{\b}\bar{\z}^{\g}\mapsto L^{\textrm{c}}_{\y}(f)=\sum_{(\b,\g)}f_{\b,\g}y_{\b,\g}.
\end{equation*}
The {\em complex moment} matrix $\M^{\textrm{c}}_{d}(\y)$ ($d\in\N$) associated with $\y$ is the matrix with rows and columns indexed by $\N^n_{d}$ such that
\begin{equation*}
\M^{\textrm{c}}_d(\y)_{\b\g}\coloneqq L^{\textrm{c}}_{\y}(\z^{\b}\bar{\z}^{\g})=y_{\b,\g}, \quad\forall\b,\g\in\N^n_{d}.
\end{equation*}
Suppose that $g=\sum_{(\b',\g')}g_{\b',\g'}\z^{\b'}\bar{\z}^{\g'}\in\C[\z,\bar{\z}]$ is a Hermitian polynomial, i.e., $\bar{g}=g$. The {\em complex localizing} matrix $\M^{\textrm{c}}_{d}(g\y)$ associated with $g$ and $\y$ is the matrix with rows and columns indexed by $\N^n_{d}$ such that
\begin{equation*}
\M^{\textrm{c}}_{d}(g\,\y)_{\b\g}\coloneqq L^{\textrm{c}}_{\y}(g\,\z^{\b}\bar{\z}^{\g})=\sum_{(\b',\g')}g_{\b',\g'}y_{\b+\b',\g+\g'}, \quad\forall\b,\g\in\N^n_{d}.
\end{equation*}
Both the complex moment matrix and the complex localizing matrix are Hermitian matrices.

Note that a distinguished difference between the real moment matrix and the complex moment matrix is that the former has the Hankel property (i.e., $\M^{\textrm{r}}_{d}(\y)_{\b\g}$ is a function of $\b+\g$), whereas the latter does not have.

There are two ways to construct a ``moment-SOS'' hierarchy for CPOP \eqref{cpop}. The first way is introducing real variables for both real and imaginary parts of each complex variable in \eqref{cpop}, i.e., letting $z_i=x_i+x_{i+n}\i$ for $i\in[n]$. Then one can convert CPOP \eqref{cpop} to a POP involving only real variables at the price of doubling the number of variables. Therefore the usual real moment-SOS hierarchy applies to the resulting real POP. In order to improve scalability, correlative and term sparsity can be exploited to yield sparsity-adapted hierarchies \cite{waki,tssos2,tssos1,tssos3}.

On the other hand, as the second way, it might be advantageous to handle CPOP \eqref{cpop} directly with the complex moment-HSOS hierarchy introduced in \cite{josz2018lasserre}.
Let $d_f\coloneqq\max\{|\a|,|\b|:f_{j,\a,\b}\ne0\}$, $d_j\coloneqq\max\{|\a|,|\b|:g_{j,\a,\b}\ne0\},j\in[m]$, and let $d_{\min}\coloneqq\max\{d_f,d_1,\ldots,d_m\}$. Then the complex moment hierarchy indexed by $d\ge d_{\min}$ (called the relaxation order) for CPOP \eqref{cpop} is given by
\begin{equation}\label{sec2-eq1}
(\textrm{Q}_{d}):\quad
\begin{cases}
\inf& L^{\textrm{c}}_{\y}(f)\\
\textrm{s.t.}&\M^{\textrm{c}}_{d}(\y)\succeq0,\\
&\M^{\textrm{c}}_{d-d_j}(g_j\y)\succeq0,\quad j\in[m],\\
&y_{\mathbf{0},\mathbf{0}}=1,
\end{cases}
\end{equation}
which is a semidefinite program (SDP) with optimum denoted by $\rho_d$. The dual of $(\textrm{Q}_{d})$ \eqref{sec2-eq1} can be formulized as the following HSOS relaxation:
\begin{equation}\label{sec2-eq2}
(\textrm{Q}_{d})^*:\quad
\begin{cases}
\sup&\rho\\
\textrm{s.t.}&f-\rho=\sigma_0+\sigma_1g_1+\ldots+\sigma_mg_m,\\
&\sigma_j\textrm{ is an HSOS},\quad j=0,\ldots,m,\\
&\deg(\sigma_0)\le2d,\deg(\sigma_jg_j)\le2d,\quad j\in[m].
\end{cases}
\end{equation}

\begin{remark}
In \eqref{sec2-eq1}, the expression ``$X\succeq0$" means an Hermitian matrix $X$ to be positive semidefinite. Since popular SDP solvers deal with only real SDPs, it is necessary to convert this condition to a condition involving only real matrices. This can be done by introducing the real part $A$ and the imaginary part $B$ of $X$ respectively such that $X=A+B\i$. Then,
\begin{equation*}
    X\succeq0\quad \iff \quad\begin{bmatrix}A&-B\\B&A
    \end{bmatrix}\succeq0.
\end{equation*}
\end{remark}

\begin{remark}
The first order moment-(H)SOS relaxation for quadratically constrained quadratic programs (QCQP) is also known as Shor's relaxation. It was proved in \cite{josz2015moment} that the real Shor's relaxation and the complex Shor's relaxation for homogeneous QCQPs yield the same bound. However, generally the complex hierarchy is weaker \revision{(i.e., producing looser bounds)} than the real hierarchy at the same relaxation order $d>1$ \revision{as Hermitian sums of squares are a special case of real sums of squares; see \cite{josz2018lasserre}}.
\end{remark}

\begin{remark}
By the complex Positivstellensatz theorem due to D’Angelo and Putinar \cite{d2009polynomial}, global convergence of the complex hierarchy is guaranteed when a sphere constraint is present.
\end{remark}

\revision{\begin{remark}
If CPOP \eqref{cpop} is feasible, then the program \eqref{sec2-eq1} is always feasible for any $d\ge d_{\min}$ as one can take the Dirac measure centering a feasible point of \eqref{cpop} which leads to a feasible point of \eqref{sec2-eq1}. Moreover, if we further assume that there is a ball/sphere constraint (alternatively, multi-ball/sphere constraints) in terms of all variables in \eqref{cpop}, then by a similar argument as for Proposition 5.8 of \cite{marshall2009}, we can show that the feasible set of \eqref{sec2-eq2} is nonempty. So the optimum of \eqref{sec2-eq1} is bounded from below by weak duality. In addition, by Lemma 3.6 of \cite{josz2015moment}, the strong duality also holds in this case.
\end{remark}}

\subsection{Sparse matrices and chordal graphs}
In this subsection, we briefly revisit the relationship between sparse matrices and chordal graphs, which is crucial for the sparsity-exploitation of this paper. For more details on sparse matrices and chordal graphs, the reader is referred to the survey \cite{va}.

An (undirected) {\em graph} $G(V,E)$ or simply $G$ consists of a set of nodes $V$ and a set of edges $E\subseteq\{\{v_i,v_j\}\mid (v_i,v_j)\in V\times V\}$. 
When $G$ is a graph, we also use $V(G)$ and $E(G)$ to indicate the node set of $G$ and the edge set of $G$, respectively. 
The {\em adjacency matrix} of $G$ is denoted by $B_G$ for which we put ones on positions corresponding to edges of $G$ as well as its diagonal and put zeros otherwise. For two graphs $G,H$, we say that $G$ is a {\em subgraph} of $H$ if $V(G)\subseteq V(H)$ and $E(G)\subseteq E(H)$, denoted by $G\subseteq H$. A graph is called a {\em chordal graph} if all its cycles of length at least four have a chord\footnote{By a chord, we means an edge that joins two nonconsecutive nodes in a cycle.}. Any non-chordal graph $G(V,E)$ can always be extended to a chordal graph $\overline{G}(V,\overline{E})$ by adding appropriate edges to $E$, which is called a {\em chordal extension} of $G(V,E)$. A {\em clique} $C\subseteq V$ of $G$ is a subset of nodes where $\{v_i,v_j\}\in E$ for any $v_i,v_j\in C$. If a clique $C$ is not a subset of any other clique, then it is called a {\em maximal clique}. It is known that maximal cliques of a chordal graph can be enumerated efficiently in linear time in the number of nodes and edges of the graph \cite{bp}.

\begin{remark}
For a graph $G$, we denote any specific chordal extension of $G$ by $\overline{G}$. The chordal extension of $G$ is generally not unique. In this paper, we will consider two particular types of chordal extensions: {\em the maximal chordal extension} and {\em approximately smallest chordal extensions}. By the maximal chordal extension, we refer to the chordal extension that completes every connected component of $G$. A chordal extension with the smallest clique number is called a {\em smallest chordal extension}. Computing a smallest chordal extension of a graph is an NP-complete problem in general. Fortunately, several heuristic algorithms, e.g., the greedy minimum degree and the greedy minimum fill-ins, are known to efficiently produce a good approximation; see \cite{treewidth} for more detailed discussions. Throughout the paper, we assume that for graphs $G,H$,
\begin{equation}\label{assum}
G\subseteq H\Longrightarrow \overline{G}\subseteq\overline{H}.
\end{equation}
This assumption is reasonable since any chordal extension of $H$ restricting to $G$ is also a chordal extension of $G$.
\end{remark}

Given a graph $G(V,E)$, a Hermitian matrix $Q$ indexed by $V=[n]$ is said to have sparsity pattern $G$ if $Q_{ij}=Q_{ji}=0$ whenever $i\ne j$ and $\{i,j\}\notin E$, i.e., $B_G\circ Q=Q$. Let $\H_G$ be the set of Hermitian matrices with sparsity pattern $G$.
A matrix in $\H_G$ exhibits a {\em block structure}. Each block corresponds to a maximal clique of $G$. The maximal block size is the maximal size of maximal cliques of $G$, namely, the \emph{clique number} of $G$. Note that there might be overlaps between blocks because different maximal cliques may share nodes.

Given a maximal clique $C$ of $G(V,E)$, we define a matrix $P_{C}\in \R^{|C|\times |V|}$ by
\begin{equation}\label{sec2-eq6}
[P_{C}]_{ij}=\begin{cases}
1, &\textrm{if }C(i)=j,\\
0, &\textrm{otherwise},
\end{cases}
\end{equation}
where $C(i)$ denotes the $i$-th node in $C$, sorted in the ordering compatible with $V$. 
Note that $P_{C}QP_{C}^T\in \H^{|C|}$ extracts a principal submatrix defined by the indices in the clique $C$ from a Hermitian matrix $Q$, and $P_{C}^TQ_{C}P_{C}$ inflates a $|C|\times|C|$ matrix $Q_{C}$ into a sparse $|V|\times |V|$ matrix.

The PSD matrices with sparsity pattern $G$ form a convex cone
\begin{equation}\label{sec2-eq5}
\H_+^{|V|}\cap\H_G=\{Q\in\H_G\mid Q\succeq0\}.
\end{equation}
When the sparsity pattern graph $G$ is chordal, the cone $\H_+^{|V|}\cap\H_G$ can be
decomposed as a sum of simple convex cones, as stated in the following theorem.
\begin{theorem}[\cite{agler}, Theorem 2.3]\label{sec2-thm}
Let $G(V,E)$ be a chordal graph and assume that $\{C_1,\ldots,C_t\}$ is the list of maximal cliques of $G(V,E)$. Then a matrix $Q\in\H_+^{|V|}\cap\H_G$ if and only if there exist $Q_{k}\in \H_+^{|C_k|}$ for $k=1,\ldots,t$ such that $Q=\sum_{k=1}^tP_{C_k}^TQ_{k}P_{C_k}$.
\end{theorem}

Given a graph $G(V,E)$, let $\Pi_{G}$ be the projection from $\H^{|V|}$ to the subspace $\H_G$, i.e., for $Q\in\H^{|V|}$,
\begin{equation}\label{sec2-eq7}
\Pi_{G}(Q)_{ij}=\begin{cases}
Q_{ij}, &\textrm{if }\{i,j\}\in E\textrm{ or }i=j,\\
0, &\textrm{otherwise}.
\end{cases}
\end{equation}

We denote by $\Pi_{G}(\H_+^{|V|})$ the set of matrices in $\H_G$ that have a PSD completion, i.e., 
\begin{equation}\label{sec2-eq8}
\Pi_{G}(\H_+^{|V|})=\{\Pi_{G}(Q)\mid Q\in\H_+^{|V|}\}.
\end{equation}

One can check that the PSD completable cone $\Pi_{G}(\H_+^{|V|})$ and the PSD cone $\H_+^{|V|}\cap\H_G$ form a pair of dual cones in $\H_G$; see \cite[Section 10.1]{va} for a proof. Moreover, for a chordal graph $G$, the decomposition result for the cone $\H_+^{|V|}\cap\H_G$ in Theorem \ref{sec2-thm} leads to the following characterization of the PSD completable cone $\Pi_{G}(\H_+^{|V|})$.
\begin{theorem}[\cite{grone1984}, Theorem 7]\label{sec2-thm2}
Let $G(V,E)$ be a chordal graph and assume that $\{C_1,\ldots,C_t\}$ is the list of maximal cliques of $G(V,E)$. Then a matrix $Q\in\Pi_{G}(\H_+^{|V|})$ if and only if $Q_{k}=P_{C_k}QP_{C_k}^T\succeq0$ for $k=1,\ldots,t$.
\end{theorem}

\section{Sparsity-adapted complex moment-HSOS hierarchies}\label{hierarchies}
In this section, we adapt the real moment-SOS hierarchies that exploit sparsity developed in \cite{waki,tssos2,tssos1,tssos3} to the complex case.

\subsection{Correlative sparsity}\label{cs}
The procedure to exploit correlative sparsity for the complex hierarchy consists of two steps: 1) partition the set of variables into subsets according to the correlations between variables emerging in the problem data, and 2) construct a sparse complex hierarchy with respect to the former partition of variables \cite{josz2018lasserre,waki}.

Let us discuss in more details. Consider the CPOP defined by \eqref{cpop}. Fix a relaxation order $d\ge d_{\min}$. Let $J'\coloneqq\{j\in[m]\mid d_j=d\}$. For $\b=(\beta_i)_i\in\N^n$, let $\supp(\b)\coloneqq\{i\in[n]\mid\beta_i\ne0\}$. We define the {\em correlative sparsity pattern (csp) graph} associated with CPOP \eqref{cpop} to be the graph $G^{\textrm{csp}}$ with nodes $V=[n]$ and edges $E$ satisfying $\{i,j\}\in E$ if one of the following holds:
\begin{enumerate}
    \item[(i)] there exists $(\b,\g)\in\supp(f)\cup\bigcup_{j\in J'}\supp(g_j)$ such that $\{i,j\}\subseteq\supp(\b)\cup\supp(\g)$;
    \item[(ii)] there exists $k\in[m]\backslash J'$ such that $\{i,j\}\subseteq\bigcup_{(\b,\g)\in\supp(g_k)}(\supp(\b)\cup\supp(\g))$.
\end{enumerate}

\begin{remark}
We adopt the idea of ``monomial sparsity" proposed in \cite{josz2018lasserre} in the definition of csp graphs, which thus is slightly different from the original definition in \cite{waki}.
\end{remark}

Let $\overline{G}^{\textrm{csp}}$ be a chordal extension of $G^{\textrm{csp}}$ and $\{I_l\}_{l\in[p]}$ be the list of maximal cliques of $\overline{G}^{\textrm{csp}}$ with $n_l\coloneqq|I_l|$. Let $\C[\z(I_l)]$ denote the ring of complex polynomials in the $n_l$ variables $\z(I_l) = \{z_i\mid i\in I_l\}$. We then partition the constraint polynomials $g_j, j\in[m]\backslash J'$ into groups $\{g_j\mid j\in J_l\}, l\in[p]$ which satisfy:
\begin{enumerate}
    \item[(i)] $J_1,\ldots,J_p\subseteq[m]\backslash J'$ are pairwise disjoint and $\cup_{l=1}^pJ_l=[m]\backslash J'$;
    \item[(ii)] for any $j\in J_l$, $\bigcup_{(\b,\g)\in\supp(g_j)}(\supp(\b)\cup\supp(\g))\subseteq I_l$, $l\in[p]$.
\end{enumerate}

\revision{\begin{example}
Consider the following CPOP
\begin{equation*}
\begin{cases}
\inf_{\z\in\C^3} &z_1\bar{z}_2+\bar{z}_1z_2+|z_3|^2\\
\st&g_1=1-|z_1|^2-|z_2|^2\ge0,\\
&g_2=1-|z_2|^2-|z_3|^2\ge0,\\
&g_3=|z_1|^4+z_2\bar{z}_3+\bar{z}_2z_3\ge0.\\
\end{cases}
\end{equation*}
Taking $d=d_{\min}=2$, we have two variable cliques $I_1=\{1,2\}, I_2=\{2,3\}$, and $J'=\{3\}, J_1=\{1\}, J_2=\{2\}$; taking $d=3$, we have one variable clique $I_1=\{1,2,3\}$, and $J'=\emptyset, J_1=\{1,2,3\}$.
\end{example}}

Next, with $l\in[p]$ and $g\in\C[\z(I_l)]$, let $\M^{\textrm{c}}_d(\y, I_l)$ (resp. $\M^{\textrm{c}}_d(g\y, I_l)$)
be the complex moment (resp. complex localizing) submatrix obtained from $\M^{\textrm{c}}_d(\y)$ (resp. $\M^{\textrm{c}}_d(g\y)$) by retaining only those rows and columns indexed by $\b\in\N_d^n$ of $\M^{\textrm{c}}_d(\y)$ (resp. $\M^{\textrm{c}}_d(g\y)$) with $\supp(\b)\subseteq I_l$.

Then, the complex (moment) hierarchy based on correlative sparsity for CPOP \eqref{cpop} is defined as
\begin{equation}\label{cs-eq1}
(\textrm{Q}^{\textrm{cs}}_{d}):\quad
\begin{cases}
\inf &L^{\textrm{c}}_{\y}(f)\\
\textrm{s.t.}&\M^{\textrm{c}}_d(\y, I_l)\succeq0,\quad l\in[p],\\
&\M^{\textrm{c}}_{d-d_j}(g_j\y, I_l)\succeq0,\quad j\in J_l, l\in[p],\\
&L^{\textrm{c}}_{\y}(g_j)\ge0,\quad j\in J',\\
&y_{\mathbf{0},\mathbf{0}}=1.
\end{cases}
\end{equation}
We denote the optimum of $(\textrm{Q}^{\textrm{cs}}_{d})$ by $\rho^{\textrm{cs}}_d$.

\begin{proposition}\label{cs-prop}
If CPOP \eqref{cpop} is a QCQP, then $(\textrm{Q}^{\textrm{cs}}_{1})$ and $(\textrm{Q}_{1})$ yield the same lower bound for \eqref{cpop}, i.e., $\rho^{\textrm{cs}}_1=\rho_1$.
\end{proposition}
\begin{proof}
By construction, the objective function and the affine constraints of $(\textrm{Q}_{1})$ involve only the decision variables $\{y_{\b,\g}\}_{(\b,\g)}$ 
with $\supp(\b)\cup\supp(\g)\subseteq I_l$ for some $l\in[p]$. Therefore, we can replace $\M^{\textrm{c}}_{1}(\y)\succeq0$ by $B_G\circ \M^{\textrm{c}}_{1}(\y)\in\Pi_{G}(\H_+^{n+1})$ without changing the optimum, where $G$ is the graph obtained from $\overline{G}^{\textrm{csp}}$ by adding a node $0$ (corresponding to $\mathbf{0}\in\N^n$) and adding edges $\{0,i\}, i\in[n]$. Note that $G$ is again a chordal graph and so the equality of optima of $(\textrm{Q}_{1})$ and $(\textrm{Q}^{\textrm{cs}}_{1})$ follows from Theorem \ref{sec2-thm2}.
\end{proof}

\subsection{Term sparsity}\label{ts}
Besides correlative sparsity, one can also exploit term sparsity for the complex hierarchy, which was recently developed for real POPs in \cite{tssos2,tssos1,tssos3}. \revision{The intuition behind this procedure is the following: starting with a minimal initial support set, one expands the support set that is taken into account by iteratively performing chordal extensions to the related sparsity pattern graphs inspired by Theorem \ref{sec2-thm2}.} We next adapt it to the complex case.

Let
$\mathscr{A} = \supp(f)\cup\bigcup_{j=1}^m\supp(g_j)$. We define the {\em term sparsity pattern (tsp) graph} at relaxation order $d$ associated with CPOP \eqref{cpop} or the set $\A$, to be the graph $G_{d}^{\textrm{tsp}}$ with nodes $V=\N^n_{d}$ and edges
\begin{equation}\label{ts-eq0}
E\coloneqq\{\{\b,\g\}\subseteq\N^n_d\mid(\b,\g)\in\A\}.
\end{equation}

\begin{remark}
There is a difference on the definitions of tsp graphs between the complex and real cases. In the real case, we use $\A\cup2\N^n_{d}$ rather than $\A$ in \eqref{ts-eq0} due to the Hankel structure of real moment matrices.
\end{remark}

\revision{\begin{example}\label{ex-ts}
Consider the following CPOP
\begin{equation*}
\begin{cases}
\inf_{\z\in\C^3} &z_1^2+\bar{z}_1^2+z_1\bar{z}_2+\bar{z}_1z_2+z_2\bar{z}_3+\bar{z}_2z_3+z_1z_2\bar{z}_3+\bar{z}_1\bar{z}_2z_3\\
\st &g_1=1-|z_1|^2-|z_2|^2-|z_3|^2\ge0.\\
\end{cases}
\end{equation*}
Figure \ref{fg-ts} illustrates the tsp graph $G_{2}^{\textrm{tsp}}$ for this CPOP, where the nodes are labeled by $\z^{\b}$ instead of $\b$ for better visualization.

\begin{figure}[htbp]
\centering
{\tiny
\begin{tikzpicture}[every node/.style={circle, draw=blue!50, thick, minimum size=8mm}]
\node (n1) at (-2,0) {$1$};
\node (n8) at (0,0) {$z_2^2$};
\node (n9) at (-2,-2) {$z_1^2$};
\node (n10) at (0,-2) {$z_3^2$};
\node (n2) at (2,0) {$z_1$};
\node (n3) at (4,0) {$z_2$};
\node (n4) at (6,0) {$z_3$};
\node (n5) at (2,-2) {$z_2z_3$};
\node (n6) at (4,-2) {$z_1z_3$};
\node (n7) at (6,-2) {$z_1z_2$};
\draw (n1)--(n9);
\draw (n2)--(n3);
\draw (n3)--(n4);
\draw (n4)--(n7);
\end{tikzpicture}}
\caption{The tsp graph with $d=2$ for Example \ref{ex-ts}}\label{fg-ts}
\end{figure} 
\end{example}}

For any graph $G$ with $V\subseteq\N^n$ and $g=\sum_{(\b',\g')}g_{\b',\g'}\z^{\b'}\bar{\z}^{\g'}\in\C[\z,\bar{\z}]$, we define the {\em $g$-support} of $G$ by
\begin{equation*}
\supp_g(G)\coloneqq\{(\b+\b',\g+\g')\mid\b=\g\in V(G)\textrm{ or }\{\b,\g\}\in E(G),(\b',\g')\in\supp(g)\}.
\end{equation*}
Let us set $d_0\coloneqq0$ and $g_0\coloneqq1$. Now assume that $G_{d,0}^{(0)}=G_{d}^{\textrm{tsp}}$ and $G_{d,j}^{(0)},j\in[m]$ are empty graphs. Then, we iteratively define an ascending chain of graphs $(G_{d,j}^{(k)}(V_{d,j},E_{d,j}^{(k)}))_{k\ge1}$ with $V_{d,j}=\N^n_{d-d_j}$ for each $j\in\{0\}\cup[m]$ by
\begin{equation}\label{ts-eq1}
G_{d,j}^{(k)}\coloneqq\overline{F_{d,j}^{(k)}},
\end{equation}
where $F_{d,j}^{(k)}$ is the graph with $V(F_{d,j}^{(k)})=\N^n_{d-d_j}$ and
\begin{equation}\label{ts-eq2}
E(F_{d,j}^{(k)})=\{\{\b,\g\}\subseteq\N^n_{d-d_j}\mid((\b,\g)+\supp(g_j))\cap(\bigcup_{i=0}^m\supp_{g_i}(G_{d,i}^{(k-1)}))\ne\emptyset\}.
\end{equation}

Let $r_j\coloneqq\binom{n+d-d_j}{d-d_j}$ for $j\in\{0\}\cup[m]$. Then with $d\ge d_{\min}$ and $k\ge1$, the complex (moment) hierarchy based on term sparsity for CPOP \eqref{cpop} is defined as
\begin{equation}\label{ts-eq3}
(\textrm{Q}^{\textrm{ts}}_{d,k}):\quad
\begin{cases}
\inf &L^{\textrm{c}}_{\y}(f)\\
\textrm{s.t.}&B_{G_{d,0}^{(k)}}\circ \M^{\textrm{c}}_d(\y)\in\Pi_{G_{d,0}^{(k)}}(\H_+^{r_0}),\\
&B_{G_{d,j}^{(k)}}\circ \M^{\textrm{c}}_{d-d_j}(g_j\y)\in\Pi_{G_{d,j}^{(k)}}(\H_+^{r_j}),\quad j\in[m],\\
&y_{\mathbf{0},\mathbf{0}}=1,
\end{cases}
\end{equation}
with optimum denoted by $\rho^{\textrm{ts}}_{d,k}$. The above hierarchy is called the (complex) {\em TSSOS} hierarchy, which is indexed by two parameters: the relaxation order $d$ and the {\em sparse order} $k$.

\begin{theorem}\label{ts-thm1}
Consider CPOP \eqref{cpop}. The following hold:
\begin{enumerate}
    \item[(i)] Fixing a relaxation order $d\ge d_{\min}$, the sequence $(\rho^{\textrm{ts}}_{d,k})_{k\ge1}$ is monotonically nondecreasing and $\rho^{\textrm{ts}}_{d,k}\le\rho_{d}$ for all $k$ (with $\rho_{d}$ defined in Section \ref{complexsos}).
    \item[(ii)] Fixing a sparse order $k\ge1$, the sequence $(\rho^{\textrm{ts}}_{d,k})_{d\ge d_{\min}}$ is monotonically nondecreasing.
\end{enumerate}
\end{theorem}
\begin{proof}
(i). For all $j,k$, by construction we have $G_{d,j}^{(k)}\subseteq G_{d,j}^{(k+1)}$, which implies that $B_{G_{d,j}^{(k)}}\circ \M^{\textrm{c}}_{d-d_j}(g_j\y)\in\Pi_{G_{d,j}^{(k)}}(\H_+^{r_j})$ is less restrictive than $B_{G_{d,j}^{(k+1)}}\circ \M^{\textrm{c}}_{d-d_j}(g_j\y)\in\Pi_{G_{d,j}^{(k+1)}}(\H_+^{r_j})$. Hence, $(\textrm{Q}_{d,k}^{\textrm{ts}})$ is a relaxation of $(\textrm{Q}_{d,k+1}^{\textrm{ts}})$ and is clearly also a relaxation of $(\textrm{Q}_{d})$. As a result, $(\rho^{\textrm{ts}}_{d,k})_{k\ge1}$ is monotonically nondecreasing and $\rho^{\textrm{ts}}_{d,k}\le\rho_{d}$ for all $k$.

(ii). The conclusion follows if we can show that the inclusion $G_{d,j}^{(k)}\subseteq G_{d+1,j}^{(k)}$ holds for all $d,j$ since this implies that $(\textrm{Q}_{d,k}^{\textrm{ts}})$ is a relaxation of $(\textrm{Q}_{d+1,k}^{\textrm{ts}})$. Let us prove $G_{d,j}^{(k)}\subseteq G_{d+1,j}^{(k)}$ by induction on $k$. For $k=1$, we have $E(G_{d}^{\textrm{tsp}})\subseteq E(G_{d+1}^{\textrm{tsp}})$ by $\eqref{ts-eq0}$, which implies $G_{d,j}^{(1)}\subseteq G_{d+1,j}^{(1)}$ for all $d,j$. Now assume that $G_{d,j}^{(k)}\subseteq G_{d+1,j}^{(k)}$ holds for all $d,j$ for a given $k\ge 1$. Then by \eqref{assum}, \eqref{ts-eq1}, \eqref{ts-eq2} and by the induction hypothesis, we deduce that $G_{d,j}^{(k+1)}\subseteq G_{d+1,j}^{(k+1)}$ holds for all $d,j$, which completes the induction.
\end{proof}

When building $(\textrm{Q}^{\textrm{ts}}_{d,k})$, we have the freedom to choose a specific chordal extension for any involved graph $G_{d,j}^{(k)}$, which offers a trade-off between the quality of obtained bounds and the computational cost. We show that if the maximal chordal extension is chosen, then with $d$ fixed, the resulting sequence of optima of the hierarchy (as $k$ increases) converges in finitely many steps to the optimum of the corresponding dense relaxation.

\begin{theorem}\label{ts-thm2}
Consider CPOP \eqref{cpop}. If the maximal chordal extension is used in \eqref{ts-eq1}, then for $d\ge d_{\min}$, $(\rho^{\textrm{ts}}_{d,k})_{k\ge 1}$ converges to $\rho_{d}$ in finitely many steps.
\end{theorem}
\begin{proof}
Let $d$ be fixed. It is clear that for all $j\in\{0\}\cup[m]$, the graph sequence $(G_{d,j}^{(k)})_{k\ge1}$ stabilizes after finitely many steps and we denote the stabilized graph by $G_{d,j}^{(\circ)}$. Let $(\textrm{Q}_{d,\circ}^{\textrm{ts}})$ be the moment relaxation corresponding to the stabilized graphs and let $\y^{*}=(y^{*}_{\b,\g})$ be an arbitrary feasible solution. Notice that $\{y_{\b,\g}\mid (\b,\g)\in\bigcup_{i=0}^m\supp_{g_j}(G_{d,j}^{(\circ)})\}$ is the set of decision variables involved in $(\textrm{Q}_{d,\circ}^{\textrm{ts}})$ and $\{y_{\b,\g}\mid (\b,\g)\in\N^n_d\times\N^n_d\}$ is the set of decision variables involved in ($\textrm{Q}_{d}$). Define $\overline{\y}^{*}=(\overline{y}^{*}_{\b,\g})_{(\b,\g)\in\N^n_d\times\N^n_d}$ as follows:
$$\overline{y}_{\b,\g}^{*}=\begin{cases}y_{\b,\g}^{*},\quad\textrm{ if }(\b,\g)\in\bigcup_{i=0}^m\supp_{g_j}(G_{d,j}^{(\circ)}),\\
0,\quad\quad\,\,\,\,\textrm{otherwise}.
\end{cases}$$
If the maximal chordal extension is used in \eqref{ts-eq1}, then we have that the matrices in $\Pi_{G_{d,j}^{(k)}}(\H_+^{r_j})$ are block-diagonal (up to permutation on rows and columns) for all $j,k$. As a consequence, $B_{G_{d,j}^{(k)}}\circ \M^{\textrm{c}}_{d-d_j}(g_j\y)\in\Pi_{G_{d,j}^{(k)}}(\H_+^{r_j})$ implies $B_{G_{d,j}^{(k)}}\circ \M^{\textrm{c}}_{d-d_j}(g_j\y)\succeq0$. 
By construction, we have $\M^{\textrm{c}}_{d-d_j}(g_j\overline{\y}^*)=B_{G_{d,j}^{(\circ)}}\circ \M^{\textrm{c}}_{d-d_j}(g_j\y^*)\succeq0$ for all $j\in\{0\}\cup[m]$. Therefore, $\overline{\y}^{*}$ is a feasible solution of ($\textrm{Q}_{d}$) and hence $L^{\textrm{c}}_{\y^{*}}(f)=L^{\textrm{c}}_{\overline{\y}^{*}}(f)\ge\rho_{d}$, which yields $\rho_{d,\circ}^{\textrm{ts}}\ge\rho_{d}$ since $\y^{*}$ is an arbitrary feasible solution of $(\textrm{Q}_{d,\circ}^{\textrm{ts}})$. 
By (i) of Theorem \ref{ts-thm1}, we already have $\rho_{d,\circ}^{\textrm{ts}}\le\rho_{d}$. So $\rho_{d,\circ}^{\textrm{ts}}=\rho_{d}$ as desired.
\end{proof}

\begin{proposition}\label{ts-prop}
If CPOP \eqref{cpop} is a QCQP, then $(\textrm{Q}^{\textrm{ts}}_{1,1})$ and $(\textrm{Q}_{1})$ yield the same lower bound for CPOP \eqref{cpop}, i.e., $\rho^{\textrm{ts}}_{1,1}=\rho_1$.
\end{proposition}
\begin{proof}
For a QCQP, $(\textrm{Q}_{1})$ reads as
\begin{equation*}
(\textrm{Q}_{1}):\quad
\begin{cases}
\inf& L^{\textrm{c}}_{\y}(f)\\
\textrm{s.t.}&\M^{\textrm{c}}_{1}(\y)\succeq0,\\
&L^{\textrm{c}}_{\y}(g_j)\ge0,\quad j\in[m],\\
&y_{\mathbf{0},\mathbf{0}}=1.
\end{cases}
\end{equation*}
Note that the objective function and the affine constraints of $(\textrm{Q}_{1})$ involve only the decision variables $\{y_{\mathbf{0},\mathbf{0}}\}\cup\{y_{\b,\g}\}_{(\b,\g)\in\A}$ with $\mathscr{A}=\supp(f)\cup\bigcup_{j=1}^m\supp(g_j)$. Hence there is no discrepancy of optima in replacing $(\textrm{Q}_{1})$ with $(\textrm{Q}^{\textrm{ts}}_{1,1})$ by construction.
\end{proof}

\noindent{\bf Sign symmetry.} Let $\A\subseteq\N^n\times\N^n$. The {\em sign symmetries} of $\A$ consist of all binary vectors $\br\in\Z_2^n\coloneqq\{0,1\}^n$ such that $\br^T(\b+\g)\equiv0$ $(\textrm{mod }2)$ for all $(\b,\g)\in\A$.
For the monomial basis $\N_d^n$, a set of sign symmetries $R=[\br_1,\ldots,\br_s]$ (regarded as a matrix with columns $\br_i\in\Z_2^n$) induces a partition on $\N_d^n$:
$\b,\g\in\N_d^n$ belong to the same block in this partition if and only if $R^T(\b+\g)\equiv\mathbf{0}$ $(\textrm{mod }2)$. 

We now prove that the block structure at each step of the TSSOS hierarchy is a refinement of the one induced by the sign symmetries of the system.
\begin{theorem}\label{ts-thm3}
Consider CPOP \eqref{cpop}. Let
$\mathscr{A} = \supp(f)\cup\bigcup_{j=1}^m\supp(g_j)$ and $R$ be its sign symmetries.
Assume $d\ge d_{\min}$ and $k\ge1$. For any $j\in\{0\}\cup[m]$ and $\b,\g\in\N^n_{d-d_j}$, if $\{\b,\g\}\in E(G_{d,j}^{(k)})$, then $R^T(\b+\g)\equiv\mathbf{0}$ $(\textrm{mod }2)$.
\end{theorem}
\begin{proof}
We prove the conclusion by induction on $k$. For $k=1$, suppose $\{\b,\g\}\in E(F_{d,j}^{(1)})$. By \eqref{ts-eq2}, there exists $(\b',\g')\in\supp(g_j)$ such that $(\b,\g)+(\b',\g')\in\A$. Hence $R^T(\b+\g)+R^T(\b'+\g')\equiv\mathbf{0}$ $(\textrm{mod }2)$ and since $R^T(\b'+\g')\equiv\mathbf{0}$ $(\textrm{mod }2)$, it follows that $R^T(\b+\g)\equiv\mathbf{0}$ $(\textrm{mod }2)$. Note that $G_{d,j}^{(1)}$ is a chordal extension of $F_{d,j}^{(1)}$ such that the additional edges connect two nodes belonging to the same connected component. Thus it is not hard to show that if $\{\b,\g\}\in E(G_{d,j}^{(k)})$, then $R^T(\b+\g)\equiv\mathbf{0}$ $(\textrm{mod }2)$. Now assume that the conclusion holds for a given $k\ge1$, which implies that for all $i$ and for any $(\b,\g)\in\supp_{g_i}(G_{d,i}^{(k)})$, $R^T(\b+\g)\equiv\mathbf{0}$ $(\textrm{mod }2)$. For $\{\b,\g\}\in E(F_{d,j}^{(k+1)})$, by \eqref{ts-eq2}, there exists $(\b',\g')\in\supp(g_j)$ such that $(\b,\g)+(\b',\g')\in\bigcup_{i=0}^m\supp_{g_i}(G_{d,i}^{(k)})$. Hence by the induction hypothesis, $R^T(\b+\g)+R^T(\b'+\g')\equiv\mathbf{0}$ $(\textrm{mod }2)$ and so $R^T(\b+\g)\equiv\mathbf{0}$ $(\textrm{mod }2)$. Because $G_{d,j}^{(k+1)}$ is a chordal extension of $F_{d,j}^{(k+1)}$ such that the additional edges connect two nodes belonging to the same connected component, we obtain that the conclusion holds for $k+1$ and complete the induction. 
\end{proof}

In contrast with the real case in which it was shown that the partition on $\N^n_{d-d_j},j\in\{0\}\cup[m]$ at the final step of the TSSOS hierarchy (i.e., when it stabilizes as $k$ increases) using the maximal chordal extension is exactly the one induced by the sign symmetries of the system \cite{tssos1}, in the complex case this is not necessarily true as the following example illustrates. 
\begin{example}
Consider the following CPOP $$\inf\{z_1+\bar{z}_1:1-z_1\bar{z}_1-z_2\bar{z}_2\ge0\}.$$
Let us take the relaxation order $d=2$. One can easily check by hand that the partition on $\N^2_{2}$ at the final step of the TSSOS hierarchy using the maximal chordal extension is $\{1,z_1,z_1^2\}$, $\{z_2^2\}$ and $\{z_2,z_1z_2\}$ while the partition induced by sign symmetries is $\{1,z_1,z_1^2,z_2^2\}$ and $\{z_2,z_1z_2\}$. Clearly, the former is a strict refinement of the latter.
\end{example}

\subsection{Correlative-term sparsity}\label{cs-ts}
We are now prepared to exploit correlative sparsity and term sparsity simultaneously in the complex hierarchy for CPOP \eqref{cpop}.

Let $\{I_l\}_{l\in[p]}, \{n_l\}_{l\in[p]}, J', \{J_l\}_{l\in[p]}$ be defined as in Section \ref{cs}. We apply the iterative procedure of exploiting term sparsity to each subsystem involving variables $\z(I_l)$ for $l\in[p]$ as follows. Let
\begin{equation}\label{cts-eq0}
    \mathscr{A} \coloneqq \supp(f)\cup\bigcup_{j=1}^m\supp(g_j)
\end{equation}
and
\begin{equation}\label{cts-eq1}
    \mathscr{A}_l \coloneqq \{(\b,\g)\in\A\mid\supp(\b)\cup\supp(\g)\subseteq I_l\}
\end{equation}
for $l\in[p]$. Fix a relaxation order $d\ge d_{\min}$. Let $G_{d,l}^{\textrm{tsp}}$ be the tsp graph with nodes $\N^{n_l}_{d-d_j}$ associated with $\A_l$ defined as in Section \ref{ts}. Note that here we embed $\N^{n_l}_{d-d_j}$ into $\N^{n}_{d-d_j}$ via the map
$\a=(\alpha_i)_{i\in I_l}\in\N^{n_l}_{d-d_j}\mapsto\a'=(\alpha'_i)_{i\in[n]}\in\N^{n}_{d-d_j}$ which satisfies 
\begin{equation*}
    \alpha'_i=\begin{cases}
    \alpha_i,\quad\textrm{if } i\in I_l,\\
    0,\quad\,\,\,\textrm{otherwise. }
    \end{cases}
\end{equation*}
Assume that $G_{d,l,0}^{(0)}=G_{d,l}^{\textrm{tsp}}$ and $G_{d,l,j}^{(0)},j\in J_l, l\in[p]$ are empty graphs. Letting
\begin{equation}\label{cts-eq2}
    \CC_{d}^{(k-1)}\coloneqq\bigcup_{l=1}^p\bigcup_{j\in \{0\}\cup J_l}\supp_{g_j}(G_{d,l,j}^{(k-1)}),\quad k\ge1,
\end{equation}
we iteratively define an ascending chain of graphs $(G_{d,l,j}^{(k)}(V_{d,l,j},E_{d,l,j}^{(k)}))_{k\ge1}$ with $V_{d,l,j}=\N^{n_l}_{d-d_j}$ for each $j\in\{0\}\cup J_l$ and each $l\in[p]$ by
\begin{equation}\label{cts-eq3}
G_{d,l,j}^{(k)}\coloneqq\overline{F_{d,l,j}^{(k)}},
\end{equation}
where $F_{d,l,j}^{(k)}$ is the graph with $V(F_{d,l,j}^{(k)})=\N^{n_l}_{d-d_j}$ and
\begin{equation}\label{cts-eq4}
E(F_{d,l,j}^{(k)})=\{\{\b,\g\}\subseteq\N^{n_l}_{d-d_j}\mid((\b,\g)+\supp(g_j))\cap\CC_{d}^{(k-1)}\ne\emptyset\}.
\end{equation}
Let $r_{d,l,j}\coloneqq\binom{n_l+d-d_j}{d-d_j}$ for all $l,j$. Then with $d\ge d_{\min}$ and $k\ge1$, the complex (moment) hierarchy based on correlative-term sparsity for CPOP \eqref{cpop} is defined as
\begin{equation}\label{cts-eq5}
(\textrm{Q}^{\textrm{cs-ts}}_{d,k}):\quad
\begin{cases}
\inf&L^{\textrm{c}}_{\y}(f)\\
\textrm{s.t.}&B_{G_{d,l,0}^{(k)}}\circ \M^{\textrm{c}}_d(\y, I_l)\in\Pi_{G_{d,l,0}^{(k)}}(\H_+^{r_{d,l,0}}),\quad l\in[p],\\
&B_{G_{d,l,j}^{(k)}}\circ \M^{\textrm{c}}_{d-d_j}(g_j\y, I_l)\in\Pi_{G_{d,l,j}^{(k)}}(\H_+^{r_{d,l,j}}),\quad j\in J_l,l\in[p],\\
&L^{\textrm{c}}_{\y}(g_j)\ge0,\quad j\in J',\\
&y_{\mathbf{0},\mathbf{0}}=1,
\end{cases}
\end{equation}
with optimum denoted by $\rho^{\textrm{cs-ts}}_{d,k}$. The above hierarchy is called the (complex) {\em CS-TSSOS} hierarchy indexed by the relaxation order $d$ and the sparse order $k$.

By similar arguments as for Theorem \ref{ts-thm1}, we can prove the following theorem.
\begin{theorem}\label{cts-thm1}
Consider CPOP \eqref{cpop}. The following hold:
\begin{enumerate}
    \item[(i)] Fixing a relaxation order $d\ge d_{\min}$, the sequence $(\rho^{\textrm{cs-ts}}_{d,k})_{k\ge1}$ is monotonically nondecreasing and $\rho^{\textrm{cs-ts}}_{d,k}\le\rho^{\textrm{cs}}_{d}$ for all $k$ (with $\rho^{\textrm{cs}}_{d}$ defined in Section \ref{cs}).
    \item[(ii)] Fixing a sparse order $k\ge 1$, the sequence $(\rho^{\textrm{cs-ts}}_{d,k})_{d\ge d_{\min}}$ is monotonically nondecreasing.
\end{enumerate}
\end{theorem}

From Theorem \ref{cts-thm1}, we have the following two-level hierarchy of lower bounds for the optimum of CPOP \eqref{cpop}:
\begin{equation}\label{cts-eq6}
\begin{matrix}
\rho^{\textrm{cs-ts}}_{d_{\min},1}&\le&\rho^{\textrm{cs-ts}}_{d_{\min},2}&\le&\cdots&\le&\rho^{\textrm{cs}}_{d_{\min}}\\
\vge&&\vge&&&&\vge\\
\rho^{\textrm{cs-ts}}_{d_{\min}+1,1}&\le&\rho^{\textrm{cs-ts}}_{d_{\min}+1,2}&\le&\cdots&\le&\rho^{\textrm{cs}}_{d_{\min}+1}\\
\vge&&\vge&&&&\vge\\
\vdots&&\vdots&&\vdots&&\vdots\\
\vge&&\vge&&&&\vge\\
\rho^{\textrm{cs-ts}}_{d,1}&\le&\rho^{\textrm{cs-ts}}_{d,2}&\le&\cdots&\le&\rho^{\textrm{cs}}_{d}\\
\vge&&\vge&&&&\vge\\
\vdots&&\vdots&&\vdots&&\vdots\\
\end{matrix}
\end{equation}

By similar arguments as for Theorem \ref{ts-thm2}, we can prove the convergence of the CS-TSSOS hierarchy at each relaxation order when the maximal chordal extension is chosen.
\begin{theorem}\label{cts-thm2}
Consider CPOP \eqref{cpop}. If the maximal chordal extension is used in \eqref{cts-eq3}, then for $d\ge d_{\min}$, $(\rho^{\textrm{cs-ts}}_{d,k})_{k\ge 1}$ converges to $\rho^{\textrm{cs}}_{d}$ in finitely many steps.
\end{theorem}

By slightly adapting the proof of Theorem \ref{ts-thm3}, one can prove that the block structure at each step of the CS-TSSOS hierarchy is also ``governed" by the sign symmetries of the system.
\begin{theorem}
Consider CPOP \eqref{cpop}. Let
$\mathscr{A} = \supp(f)\cup\bigcup_{j=1}^m\supp(g_j)$ and $R$ be its sign symmetries.
Assume $d\ge d_{\min}$ and $k\ge1$. For any $j\in J_l$, $l\in[p]$ and $\b,\g\in\N^{n_l}_{d-d_j}$, if $\{\b,\g\}\in E(G_{d,l,j}^{(k)})$, then $R^T(\b+\g)\equiv\mathbf{0}$ $(\textrm{mod }2)$.
\end{theorem}

If CPOP \eqref{cpop} is a QCQP, then by Proposition \ref{cs-prop}, we have $\rho^{\textrm{cs}}_1=\rho_1$. To ensure any higher order relaxation $(\textrm{Q}^{\textrm{cs-ts}}_{d,k})$ ($d>1$) achieves a better lower bound than Shor's relaxation, we may add an extra first order moment matrix for each variable clique\footnote{Even if CPOP \eqref{cpop} is not a QCQP, this operation could also strengthen the relaxation.}:
\begin{equation}\label{cts-eq7}
(\textrm{Q}^{\textrm{cs-ts}}_{d,k})':\quad
\begin{cases}
\inf&L^{\textrm{c}}_{\y}(f)\\
\textrm{s.t.}&B_{G_{d,l,0}^{(k)}}\circ \M^{\textrm{c}}_d(\y, I_l)\in\Pi_{G_{d,l,0}^{(k)}}(\H_+^{r_{d,l,0}}),\quad l\in[p],\\
&\M^{\textrm{c}}_{1}(\y, I_l)\succeq0,\quad l\in[p],\\
&B_{G_{d,l,j}^{(k)}}\circ \M^{\textrm{c}}_{d-d_j}(g_j\y, I_l)\in\Pi_{G_{d,l,j}^{(k)}}(\H_+^{r_{d,l,j}}),\quad j\in J_l,l\in[p],\\
&L^{\textrm{c}}_{\y}(g_j)\ge0,\quad j\in J',\\
&y_{\mathbf{0},\mathbf{0}}=1.
\end{cases}
\end{equation}

\section{The minimum initial relaxation step of the complex hierarchy}\label{initial}
For CPOP \eqref{cpop}, suppose that $f$ is not homogeneous or the constraint polynomials $g_j,j\in[m]$ have different degrees. Then it might be beneficial to assign different relaxation orders to different subsystems obtained from the correlative sparsity pattern for the initial relaxation step of the complex hierarchy instead of using the uniform minimum relaxation order $d_{\min}$. More specifically, we redefine the csp graph $G^{\textrm{icsp}}(V,E)$ as follows: let $V=[n]$ and $\{i,j\}\in E$ if there exists $(\b,\g)\in\supp(f)\cup\bigcup_{j\in[m]}\supp(g_j)$ such that $\{i,j\}\subseteq\supp(\b)\cup\supp(\g)$. This is clearly a subgraph of $G^{\textrm{csp}}$ defined in Section \ref{cs} and hence typically has a smaller chordal extension. Let $\overline{G}^{\textrm{icsp}}$ be a chordal extension of $G^{\textrm{icsp}}$ and $\{I_l\}_{l\in[p]}$ be the list of maximal cliques of $\overline{G}^{\textrm{icsp}}$ with
$n_l\coloneqq|I_l|$. Now we partition the constraint polynomials $g_j,j\in[m]$ into groups $\{g_j\mid j\in J_l\}_{l\in[p]}$ and $\{g_j\mid j\in J'\}$ which satisfy:
\begin{enumerate}
    \item[(i)] $J_1,\ldots,J_p,J'\subseteq[m]$ are pairwise disjoint and $\bigcup_{l=1}^pJ_l\cup J'=[m]$;
    \item[(ii)] for any $j\in J_l$, $\bigcup_{(\b,\g)\in\supp(g_j)}(\supp(\b)\cup\supp(\g))\subseteq I_l$, $l\in[p]$;
     \item[(iii)] for any $j\in J'$, $\bigcup_{(\b,\g)\in\supp(g_j)}(\supp(\b)\cup\supp(\g))\nsubseteq I_l$ for all $l\in[p]$.
\end{enumerate}

Assume that $f$ decomposes as $f=\sum_{l\in[p]}f_l$ such that $\bigcup_{(\b,\g)\in\supp(f_l)}(\supp(\b)\cup\supp(\g))\subseteq I_l$ for $l\in[p]$. We define the vector of minimum relaxation orders  $\o=(o_l)_l\in\N^{p}$ with $o_l\coloneqq\max(\{d_j:j\in J_l\}\cup\{\lceil\deg(f_l)/2\rceil\})$. Then with $k\ge1$, we consider the following initial relaxation step of the complex hierarchy:
\begin{equation}\label{imom-sos}
(\textrm{Q}^{\textrm{cs-ts}}_{\textrm{min},k}):\quad
\begin{cases}
\inf &L^{\textrm{c}}_{\y}(f)\\
\textrm{s.t.}&B_{G_{o_l,l,0}^{(k)}}\circ \M^{\textrm{c}}_{o_l}(\y, I_l)\in\Pi_{G_{o_l,l,0}^{(k)}}(\H_+^{s_{l,0}}),\quad l\in[p],\\
&\M^{\textrm{c}}_{1}(\y, I_l)\succeq0,\quad l\in[p],\\
&B_{G_{o_l,l,j}^{(k)}}\circ \M^{\textrm{c}}_{o_l-d_j}(g_j\y, I_l)\in\Pi_{G_{o_l,l,j}^{(k)}}(\H_+^{s_{l,j}}),\quad j\in J_l, l\in[p],\\
&L^{\textrm{c}}_{\y}(g_j)\ge0,\quad j\in J',\\
&y_{\mathbf{0},\mathbf{0}}=1,
\end{cases}
\end{equation}
where the sparsity pattern graphs $G_{o_l,l,j}^{(k)},j\in J_l,l\in[p]$ are defined in the same manner as in Section \ref{cs-ts} with $V(G_{o_l,l,j}^{(k)})=\N^{n_l}_{o_l-d_j}$ and $s_{l,j}\coloneqq\binom{n_l+o_l-d_j}{o_l-d_j}$ for all $l,j$.
\begin{remark}
Similarly, we can also define the minimum initial relaxation step for the real hierarchy.
\end{remark}

\section{Numerical experiments}\label{experiments}
In this section, we present numerical results of the proposed sparsity-adapted complex hierarchies for complex polynomial optimization problems. The hierarchies were implemented in the Julia package {\tt TSSOS} \cite{magron2021tssos}, which utilizes the Julia packages {\tt LightGraphs} \cite{graph} to handle graphs, {\tt ChordalGraph} \cite{Wang20} to generate approximately smallest chordal extensions, and {\tt JuMP} \cite{jump} to model the SDP. For the numerical experiments in this paper, {\tt Mosek} \cite{mosek} is used as an SDP solver. {\tt TSSOS} is freely available at

\vspace{2pt}
\centerline{\href{https://github.com/wangjie212/TSSOS}{https://github.com/wangjie212/TSSOS}.}
\vspace{2pt}

All numerical examples were computed on an Intel Core i5-8265U@1.60GHz CPU with 8GB RAM memory. We list the notations that are used in this section in Table \ref{table1}. The running times include the time for pre-processing (to get the block structure), the time for modeling the SDP and the time for solving the SDP. 

\begin{table}[htbp]
\caption{The notation}\label{table1}
\centering
\begin{tabular}{|c|c|}
\hline
$n$&number of complex variables\\
\hline
$d$&relaxation order\\
\hline
$k$&sparse order\\
\hline
mb&maximal size of PSD blocks\\
\hline
opt&optimum\\
\hline
time&running time in seconds\\ 
\hline
gap&optimality gap\\
\hline
-&an out of memory error\\
\hline
\end{tabular}
\end{table}

\subsection{Randomly generated examples}
Given an integer $p\in\N$, we randomly generate a CPOP as follows: 
\begin{enumerate}
    \item Let $f=\sum_{l=1}^pf_l\in\C_{4}[z_{1},\ldots,z_{5(p+1)},\bar{z}_{1},\ldots,\bar{z}_{5(p+1)}]$, where for all $l\in[p]$, $f_l=\bar{f}_l\in\C_{4}[z_{5(l-1)+1},\ldots,z_{5(l-1)+10},\bar{z}_{5(l-1)+1},\ldots,\bar{z}_{5(l-1)+10}]$ is a polynomial with $40$ terms whose real/imaginary parts of coefficients are selected with respect to the uniform probability distribution on $[-1,1]$;
    \item Let us encode multi-ball constraints with $g_l=1-\sum_{i=1}^{10}z_{5(l-1)+i}\bar{z}_{5(l-1)+i}$ for $l\in[p]$;
    \item The CPOP is defined as $\inf_{\z\in\C^{5(p+1)}}\{f(\z,\bar{\z}):g_l(\z,\bar{\z})\ge0,l=1,\ldots,p\}$.
\end{enumerate}
In such a way, we generate $10$ random CPOPs with $p=10,20,\ldots,100$, respectively. We compare the performance of the complex hierarchy with that of the real hierarchy in solving these instances with $d=2,k=1$ and with approximately smallest chordal extensions or the maximal chordal extension. The results are displayed in Table \ref{random2}. The column ``CE" indicates which type of chordal extensions we use: ``min" represents approximately smallest chordal extensions and ``max" represents the maximal chordal extension.


\begin{table}[htbp]
\caption{The complex hierarchy versus the real hierarchy with $d=2,k=1$ for minimizing quartic objective functions on multi-balls}\label{random2}
\centering
\small
\begin{tabular}{|c|c|c|c|c|c|c|c|}
\hline
\multirow{2}*{$n$}&\multirow{2}*{CE}&\multicolumn{3}{c|}{Complex}&\multicolumn{3}{c|}{Real}\\
\cline{3-8}
&&mb&opt&time&mb&opt&time\\
\hline
\multirow{2}*{55}&min&6&-24.6965&0.95&21&-21.2240&9.21\\
\cline{2-8}
&max&36&-24.4543&5.82&-&-&-\\
\hline
\multirow{2}*{105}&min&6&-48.9783&2.45&21&-40.4650&36.5\\
\cline{2-8}
&max&46&-48.8367&16.9&-&-&-\\
\hline
\multirow{2}*{155}&min&6&-69.4493&4.10&21&-57.9840&44.3\\
\cline{2-8}
&max&52&-69.2345&29.3&-&-&-\\
\hline
\multirow{2}*{205}&min&6&-100.132&6.92&21&-82.7737&85.5\\
\cline{2-8}
&max&54&-99.4924&65.4&-&-&-\\
\hline
\multirow{2}*{255}&min&8&-122.621&10.9&21&-102.728&107\\
\cline{2-8}
&max&46&-121.754&81.7&-&-&-\\
\hline
\multirow{2}*{305}&min&8&-151.096&13.7&21&-126.094&134\\
\cline{2-8}
&max&52&-149.406&54.7&-&-&-\\
\hline
\multirow{2}*{355}&min&8&-172.275&18.3&21&-144.936&182\\
\cline{2-8}
&max&52&-170.655&65.2&-&-&-\\
\hline
\multirow{2}*{405}&min&8&-197.036&24.7&21&-163.360&229\\
\cline{2-8}
&max&48&-195.287&79.6&-&-&-\\
\hline
\multirow{2}*{455}&min&8&-224.471&29.3&21&-184.701&278\\
\cline{2-8}
&max&52&-222.837&85.5&-&-&-\\
\hline
\multirow{2}*{505}&min&8&-238.760&35.2&21&-199.774&357\\
\cline{2-8}
&max&46&-237.437&108&-&-&-\\
\hline
\end{tabular}
\end{table}

As we can see from Table \ref{random2}, for the complex hierarchy, the bound with the maximal chordal extension is slightly better than the bound with approximately smallest chordal extensions. Both of them have relative gaps around $20\%$ with respect to the bound given by the real hierarchy with approximately smallest chordal extensions. The gap can be reduced if we increase the sparse order of the complex hierarchy. For instance, for the random CPOP with $p=10$, the complex hierarchy with $d=2,k=1$ has a relative gap of $16.36\%$ and the complex hierarchy with $d=2,k=2$ has a relative gap of $0.91\%$; see Table \ref{random3}.

\begin{table}[htbp]
\caption{The complex hierarchy with different sparse orders. Here, ``gap" is the relative gap with respect to the bound given by the real hierarchy in Table \ref{random2}.}\label{random3}
\centering
\begin{tabular}{|c|c|c|c|c|c|c|}
\hline
$n$&$k$&CE&mb&opt&time&gap\\
\hline
55&1&min&6&-24.6965&0.95&16.36\%\\
\hline
55&2&min&38&-21.4174&13.3&0.91\%\\
\hline
\end{tabular}
\end{table}

On the other hand, when using approximately smallest chordal extensions and with $k=1$, the complex hierarchy is over $9$ times faster then the real hierarchy for each instance. Due to the limitation of memory, the real hierarchy with the maximal chordal extension is unsolvable.

\subsection{AC-OPF instances}
The AC optimal power flow (AC-OPF) is a central problem in power systems, \revision{which aims to minimize the generation cost of an alternating current transmission network under the physical constraints (Kirchhoff’s laws, Ohm’s law, and power balance equations) as well as operational constraints.} Mathematically, it can be formulated as the following CPOP:
\begin{equation}\label{opf}
\begin{cases}
\inf\limits_{V_i,S_s^g}&\sum_{s\in G}(\mathbf{c}_{2s}(\Re(S_{s}^g))^2+\mathbf{c}_{1s}\Re(S_{s}^g)+\mathbf{c}_{0s})\\
\,\textrm{s.t.}&\angle V_r=0,\\
&\mathbf{S}_{s}^{gl}\le S_{s}^{g}\le \mathbf{S}_{s}^{gu},\quad\forall s\in G,\\
&\boldsymbol{\upsilon}_{i}^l\le|V_i|\le \boldsymbol{\upsilon}_{i}^u,\quad\forall i\in N,\\
&\sum_{s\in G_i}S_s^g-\mathbf{S}_i^d-\mathbf{Y}_i^{sh}|V_{i}|^2=\sum_{(i,j)\in E_i\cup E_i^R}S_{ij},\quad\forall i\in N,\\
&S_{ij}=(\mathbf{Y}_{ij}^*-\mathbf{i}\frac{\mathbf{b}_{ij}^c}{2})\frac{|V_i|^2}{|\mathbf{T}_{ij}|^2}-\mathbf{Y}_{ij}^*\frac{V_iV_j^*}{\mathbf{T}_{ij}},\quad\forall (i,j)\in E,\\
&S_{ji}=(\mathbf{Y}_{ij}^*-\mathbf{i}\frac{\mathbf{b}_{ij}^c}{2})|V_j|^2-\mathbf{Y}_{ij}^*\frac{V_i^*V_j}{\mathbf{T}_{ij}^*},\quad\forall (i,j)\in E,\\
&|S_{ij}|\le \mathbf{s}_{ij}^u,\quad\forall (i,j)\in E\cup E^R,\\
&\boldsymbol{\theta}_{ij}^{\Delta l}\le \angle (V_i V_j^*)\le \boldsymbol{\theta}_{ij}^{\Delta u},\quad\forall (i,j)\in E,\\
\end{cases}
\end{equation}
where $V_i$ is the voltage, $S_s^{g}$ is the power generation, $S_{ij}$ is the power flow (all are complex variables; $\Re(\cdot)$ and $\angle\cdot$ stand for the real part and the angle of a complex number, respectively) and all symbols in boldface are constants \revision{($\mathbf{Y}$: admittance, $\boldsymbol{\theta}^{\Delta}$: voltage angle difference limit)}. Notice that $G$ is the collection of generators and $N$ is the collection of buses. For a full description on the AC-OPF problem, the reader may refer to \cite{baba2019} as well as \cite{bienstock2020}. 
We will consider AC-OPF instances for which the following assumption holds.
\vspace{0.5em}

\noindent{\bf AC-OPF Assumption.} There is at most one generator attached to each bus, i.e., $|G_i|\le1$ for all $i\in N$. In this case we assume that a generator $s\in G$ is attached to the bus $i_s\in N$.
\vspace{0.5em}

If {\bf AC-OPF Assumption} holds, then we can eliminate the variables $S_s^g$ from \eqref{opf} to obtain a CPOP involving only the voltage $V_i$:
\begin{equation}\label{opf-simple}
\begin{cases}
\inf\limits_{V_i}&\sum_{s\in G}(\mathbf{c}_{2s}\Re(\mathbf{S}_{i_s}^d+\mathbf{Y}_{i_s}^{sh}|V_{i_s}|^2+\sum_{(i_s,j)\in E_{i_s}\cup E_{i_s}^R}S_{i_sj})^2\\
&\quad\quad\quad+\mathbf{c}_{1s}\Re(\mathbf{S}_{i_s}^d+\mathbf{Y}_{i_s}^{sh}|V_{i_s}|^2+\sum_{(i_s,j)\in E_{i_s}\cup E_{i_s}^R}S_{i_sj})+\mathbf{c}_{0s})\\
\textrm{s.t.}&\angle V_r=0,\\
&\mathbf{S}_{s}^{gl}\le \mathbf{S}_{i_s}^d+\mathbf{Y}_{i_s}^{sh}|V_{i_s}|^2+\sum_{(i_s,j)\in E_{i_s}\cup E_{i_s}^R}S_{i_sj}\le \mathbf{S}_{s}^{gu},\quad\forall s\in G,\\
&\boldsymbol{\upsilon}_{i}^l\le|V_i|\le \boldsymbol{\upsilon}_{i}^u,\quad\forall i\in N,\\
&S_{ij}=(\mathbf{Y}_{ij}^*-\mathbf{i}\frac{\mathbf{b}_{ij}^c}{2})\frac{|V_i|^2}{|\mathbf{T}_{ij}|^2}-\mathbf{Y}_{ij}^*\frac{V_iV_j^*}{\mathbf{T}_{ij}},\quad\forall (i,j)\in E,\\
&S_{ji}=(\mathbf{Y}_{ij}^*-\mathbf{i}\frac{\mathbf{b}_{ij}^c}{2})|V_j|^2-\mathbf{Y}_{ij}^*\frac{V_i^*V_j}{\mathbf{T}_{ij}^*},\quad\forall (i,j)\in E,\\
&|S_{ij}|\le \mathbf{s}_{ij}^u,\quad\forall (i,j)\in E\cup E^R,\\
&\boldsymbol{\theta}_{ij}^{\Delta l}\le \angle (V_i V_j^*)\le \boldsymbol{\theta}_{ij}^{\Delta u},\quad\forall (i,j)\in E.\\
\end{cases}
\end{equation}

Note that in \eqref{opf-simple}, if we substitute the expression of $S_{ij}$ for $S_{ij}$ in $|S_{ij}|\le\mathbf{s}_{ij}^u$, i.e., $S_{ij}\bar{S}_{ij}\le(\mathbf{s}_{ij}^u)^2$, we actually get a quartic constraint. To implement Shor's relaxation for QCQPs, we then relax it to a quadratic constraint by using the trick described in \cite{bienstock2020}. The minimum initial relaxation step of the complex hierarchy $(\textrm{Q}^{\textrm{cs-ts}}_{\textrm{min},1})$ for \eqref{opf-simple} is able to provide a tighter lower bound than Shor's relaxation, which we thereby refer to as the 1.5th order relaxation.

To tackle an AC-OPF instance, we first compute a locally optimal solution with a local solver and then rely on lower bounds obtained from SDP relaxations to certify global optimality. Suppose that the optimum reported by the local solver is AC and the optimum of a certain SDP relaxation is opt. The {\em optimality gap} between the locally optimal solution and the SDP relaxation is defined by
\begin{equation*}
    \textrm{gap}\coloneqq\frac{\textrm{AC}-\textrm{opt}}{\textrm{AC}}\times100\%.
\end{equation*}
If the optimality gap is less than $1\%$, then we accept the locally optimal solution to be globally optimal. 

We select instances satisfying {\bf AC-OPF  Assumption} from the AC-OPF library {\em \href{https://github.com/power-grid-lib/pglib-opf}{PGLiB}} \cite{baba2019}. The number appearing in each instance name stands for the number of buses, which is equal to the number of complex variables involved in \eqref{opf-simple}. 
For these instances, we compute the complex hierarchy as well as the real hierarchy with $k=1$ and with the maximal chordal extension. The results are displayed in Table \ref{ac-opf1}, Table \ref{ac-opf2} and Table \ref{ac-opf3} in which the column ``Order" indicates the relaxation order.

As we can see from the tables, for Shor's relaxation (the 1st order relaxation), the complex hierarchy and the real hierarchy give the same lower bound (up to a given precision) while the complex hierarchy is slightly faster. For the 1.5th order relaxation, the complex hierarchy typically gives a looser bound than the real hierarchy whereas it is faster than the real hierarchy by a factor of $1\sim8$. Shor's relaxation is able to certify global optimality for 19 out of all 36 instances. For the remaining 17 instances, with the 1.5th order relaxation the complex hierarchy is able to certify global optimality for 6 instances and the real hierarchy is able to certify global optimality for 11 instances.

\begin{table}[htbp]
\caption{The results for AC-OPF instances: typical operating conditions}\label{ac-opf1}
\centering
\small
\begin{tabular}{|c|c|c|c|c|c|c|c|c|c|}
\hline
\multirow{2}*{Case}&\multirow{2}*{Order}&\multicolumn{4}{c|}{Complex}&\multicolumn{4}{c|}{Real}\\
\cline{3-10}
&&mb&opt&time&gap&mb&opt&time&gap\\
\hline
14\_ieee&1st&6&$2.1781\text{e}3$&0.07&$0.00\%$&6&$2.1781\text{e}3$&0.09&$0.00\%$\\
\hline
\multirow{2}*{30\_ieee}&1st&8&$7.5472\text{e}3$&0.12&$8.06\%$&8&$7.5472\text{e}3$&0.15&$8.06\%$\\
\cline{2-10}
&1.5th&12&$8.2073\text{e}3$&0.66&$0.02\%$&22&$8.2085\text{e}3$&0.97&$0.00\%$\\
\hline
\multirow{2}*{39\_epri}&1st&8&$1.3565\text{e}4$&0.17&$2.00\%$&8&$1.3565\text{e}4$&0.22&$2.00\%$\\
\cline{2-10}
&1.5th&14&$1.3765\text{e}4$&1.08&$0.55\%$&25&$1.3842\text{e}4$&1.12&$0.00\%$\\
\hline
57\_ieee&1st&12&$3.7588\text{e}4$&0.27&0.00\%&12&$3.7588\text{e}4$&0.32&0.00\%\\
\hline
\multirow{2}*{89\_pegase}&1st&24&1.0670e5&0.72&0.55\%&24&1.0670e5&0.74&0.55\%\\
\cline{2-10}
&1.5th&96&$1.0709\text{e}5$&263&0.19\%&184&$1.0715\text{e}5$&1232&0.13\%\\
\hline
\multirow{2}*{118\_ieee}&1st&10&9.6900e4&0.49&0.32\%&10&9.6901e4&0.57&0.32\%\\
\cline{2-10}
&1.5th&20&$9.7199\textrm{e}4$&5.22&0.02\%&37&9.7214e4&8.78&0.00\%\\
\hline
\multirow{3}*{162\_ieee\_dtc}&1st&28&$1.0164\text{e}5$&1.49&$5.96\%$&28&$1.0164\text{e}5$&1.51&$5.96\%$\\
\cline{2-10}
&1.5th&40&$1.0249\text{e}5$&17.1&$5.17\%$&74&$1.0645\text{e}5$&87.5&$1.51\%$\\
\cline{2-10}
&2nd&146&-&-&-&282&-&-&-\\
\hline
\multirow{2}*{179\_goc}&1st&10&7.5016e5&0.72&0.55\%&10&7.5016e5&0.77&0.55\%\\
\cline{2-10}
&1.5th&20&$7.5078\text{e}5$&6.77&0.46\%&37&$7.5382\text{e}5$&10.6&0.06\%\\
\hline
\multirow{2}*{300\_ieee}&1st&14&5.5424e5&1.41&1.94\%&16&5.5424e5&1.49&1.94\%	\\
\cline{2-10}
&1.5th&22&$5.6455\text{e}5$&19.1&0.12\%&40&$5.6522\text{e}5$&27.3&0.00\%\\
\hline
\multirow{2}*{1354\_pegase}&1st&26&1.2172e6&10.9&3.30\%&26&1.2172e6&13.1&3.30\%\\
\cline{2-10}
&1.5th&26&$1.2304\text{e}6$&255&2.29\%&49&1.2514e6&392&0.59\%\\
\hline
2383wp\_k&1st&48&$1.8620\text{e}6$&36.2&0.33\%&50&1.8617e6&43.4&0.35\%\\
\hline
\multirow{2}*{2869\_pegase}&1st&26&2.4387e6&47.2&0.98\%&26&2.4388e6&67.3&0.97\%\\
\cline{2-10}
&1.5th&98&$2.4586\text{e}6$&1666&0.17\%&191&-&-&-\\
\hline
\end{tabular}
\end{table}

\begin{table}[htbp]
\caption{The results for AC-OPF instances: congested operating conditions }\label{ac-opf2}
\centering
\small
\begin{tabular}{|c|c|c|c|c|c|c|c|c|c|}
\hline
\multirow{2}*{Case}&\multirow{2}*{Order}&\multicolumn{4}{c|}{Complex}&\multicolumn{4}{c|}{Real}\\
\cline{3-10}
&&mb&opt&time&gap&mb&opt&time&gap\\
\hline
\multirow{2}*{14\_ieee}&1st&6&5.6886e3&0.06&$5.18\%$&6&5.6886e3&0.07&$5.18\%$\\
\cline{2-10}
&1.5th&14&5.9981e3&0.39&0.02\%&22&5.9994e3&0.52&0.00\%\\
\hline
\multirow{2}*{30\_ieee}&1st&8&1.7253e4&0.13&$4.40\%$&8&1.7253e4&0.14&$4.40\%$\\
\cline{2-10}
&1.5th&12&1.7941e4&0.78&$0.59\%$&22&1.8044e4&1.09&$0.00\%$\\
\hline
\multirow{2}*{39\_epri}&1st&8&2.4523e5&0.19&$1.78\%$&8&2.4523e5&0.21&$1.78\%$\\
\cline{2-10}
&1.5th&14&2.4707e5&1.08&$1.04\%$&25&2.4966e5&1.70&$0.00\%$\\
\hline
57\_ieee&1st&12&4.9289e5&0.22&0.00\%&12&4.9289e5&0.29&0.00\%\\
\hline
\multirow{3}*{89\_pegase}&1st&24&1.0052e5&0.68&22.78\%&24&1.0052e5&0.75&22.78\%\\
\cline{2-10}
&1.5th&96&$1.0145\text{e}5$&196&22.06\%&184&$1.0322\text{e}5$&1448&20.71\%\\
\cline{2-10}
&2nd&224&-&-&-&426&-&-&-\\
\hline
\multirow{3}*{118\_ieee}&1st&10&1.9375e5&0.58&20.02\%&10&1.9375e5&0.65&20.02\%\\
\cline{2-10}
&1.5th&20&$2.0193\textrm{e}5$&6.15&16.64\%&37&2.2286e5&8.58&9.00\%\\
\cline{2-10}
&2nd&48&$2.2318\textrm{e}5$&125&7.87\%&92&-&-&-\\
\hline
\multirow{3}*{162\_ieee\_dtc}&1st&28&$1.1206\text{e}5$&1.67&$7.38\%$&28&$1.1206\text{e}5$&1.71&$7.38\%$\\
\cline{2-10}
&1.5th&40&$1.1284\text{e}5$&22.5&$6.74\%$&74&$1.1955\text{e}5$&81.4&$1.19\%$\\
\cline{2-10}
&2nd&146&-&-&-&282&-&-&-\\
\hline
\multirow{2}*{179\_goc}&1st&10&1.7224e6&0.77&10.85\%&10&1.7224e6&0.78&10.85\%\\
\cline{2-10}
&1.5th&20&$1.8438\text{e}6$&7.75&4.57\%&37&$1.9226\text{e}6$&8.57&0.48\%\\
\hline
\multirow{2}*{300\_ieee}&1st&14&6.7932e5&1.22&0.83\%&16&6.7932e5&1.69&0.83\%	\\
\cline{2-10}
&1.5th&22&$6.8411\text{e}5$&19.4&0.13\%&40&$6.8493\text{e}5$&25.2&0.01\%\\
\hline
\multirow{2}*{1354\_pegase}&1st&26&1.4871e6&9.55&0.75\%&26&1.4872e6&12.7&0.74\%\\
\cline{2-10}
&1.5th&26&$1.4929\text{e}6$&265&0.36\%&49&1.4929e6&379&0.36\%\\
\hline
2383wp\_k&1st&48&2.7913e5&36.0&0.00\%&50&2.7913e5&39.3&0.00\%\\
\hline
\multirow{2}*{2869\_pegase}&1st&26&2.9059e6&44.8&0.81\%&26&2.9059e6&67.8&0.81\%\\
\cline{2-10}
&1.5th&98&$2.9262\text{e}6$&1918&0.12\%&191&-&-&-\\
\hline
\end{tabular}
\end{table}

\begin{table}[htbp]
\caption{The results for AC-OPF instances: small angle difference conditions}\label{ac-opf3}
\centering
\small
\begin{tabular}{|c|c|c|c|c|c|c|c|c|c|}
\hline
\multirow{2}*{Case}&\multirow{2}*{Order}&\multicolumn{4}{c|}{Complex}&\multicolumn{4}{c|}{Real}\\
\cline{3-10}
&&mb&opt&time&gap&mb&opt&time&gap\\
\hline
14\_ieee&1st&6&$2.7743\text{e}3$&0.10&$0.09\%$&6&$2.7743\text{e}3$&0.10&$0.09\%$\\
\hline
\multirow{2}*{30\_ieee}&1st&8&$7.5472\text{e}3$&0.16&$8.06\%$&8&$7.5472\text{e}3$&0.17&$8.06\%$\\
\cline{2-10}
&1.5th&12&$8.2072\text{e}3$&0.83&$0.02\%$&22&$8.2085\text{e}3$&0.94&$0.00\%$\\
\hline
\multirow{2}*{39\_epri}&1st&8&$1.4791\text{e}4$&0.14&$0.29\%$&8&$1.4791\text{e}4$&0.18&$0.29\%$\\
\cline{2-10}
&1.5th&14&$1.4831\text{e}4$&0.91&$0.02\%$&25&$1.4832\text{e}4$&1.16&$0.01\%$\\
\hline
57\_ieee&1st&12&$3.8646\text{e}4$&0.23&0.04\%&12&$3.8646\text{e}4$&0.32&0.04\%\\
\hline
\multirow{2}*{89\_pegase}&1st&24&1.0672e5&0.61&0.54\%&24&1.0672e5&0.75&0.54\%\\
\cline{2-10}
&1.5th&96&$1.0700\text{e}5$&159&0.27\%&184&$1.0713\text{e}5$&1201&0.15\%\\
\hline
\multirow{3}*{118\_ieee}&1st&10&1.0191e5&0.48&3.10\%&10&1.0191e5&0.56&3.10\%\\
\cline{2-10}
&1.5th&20&$1.0239\textrm{e}5$&5.52&2.64\%&37&1.0336e5&8.56&1.71\%\\
\cline{2-10}
&2nd&48&$1.0401\textrm{e}5$&120&1.10\%&92&-&-&-\\
\hline
\multirow{3}*{162\_ieee\_dtc}&1st&28&$1.0283\text{e}5$&1.39&$5.39\%$&28&$1.0283\text{e}5$&1.57&$5.39\%$\\
\cline{2-10}
&1.5th&40&$1.0434\text{e}5$&20.6&$4.00\%$&74&$1.0740\text{e}5$&69.1&$1.19\%$\\
\cline{2-10}
&2nd&146&-&-&-&282&-&-&-\\
\hline
\multirow{2}*{179\_goc}&1st&10&7.5261e5&0.75&1.30\%&10&7.5261e5&1.12&1.30\%\\
\cline{2-10}
&1.5th&20&$7.5361\text{e}5$&7.43&1.17\%&37&$7.5583\text{e}5$&7.49&0.88\%\\
\hline
\multirow{2}*{300\_ieee}&1st&14&5.6162e5&1.25&0.72\%&16&5.6168e5&1.81&0.71\%	\\
\cline{2-10}
&1.5th&22&$5.6557\text{e}5$&20.9&0.02\%&40&$5.6572\text{e}5$&22.8&0.00\%\\
\hline
\multirow{2}*{1354\_pegase}&1st&26&1.2172e6&10.4&3.30\%&26&1.2172e6&13.1&3.30\%\\
\cline{2-10}
&1.5th&26&$1.2358\text{e}6$&259&1.83\%&49&1.2586e6&260&0.02\%\\
\hline
2383wp\_k&1st&48&$1.9060\text{e}6$&35.5&0.27\%&50&1.9061e6&48.6&0.27\%\\
\hline
\multirow{2}*{2869\_pegase}&1st&26&2.4488e6&40.2&0.81\%&26&2.4490e6&57.1&0.80\%\\
\cline{2-10}
&1.5th&96&$2.4495\text{e}6$&1879&0.78\%&191&-&-&-\\
\hline
\end{tabular}
\end{table}

\section{Conclusions}\label{cons}
In this paper, we have studied the sparsity-adapted complex hierarchy for complex polynomial optimization problems by taking into account both correlative and term sparsity. Numerical experiments demonstrate that the complex hierarchy offers a trade-off between the computational efficiency and the quality of obtained bounds relative to the real hierarchy. We hope that the sparsity-adapted complex hierarchy could help to tackle large-scale CPOPs arising either from the academic literature or from practical industrial problems.

~

\paragraph{\textbf{Acknowledgements}.} 
Both authors were supported by the Tremplin ERC Stg Grant ANR-18-ERC2-0004-01 (T-COPS project).
The second author was supported by the FMJH Program PGMO (EPICS project), as well as the PEPS2 Program (FastOPF project) funded by AMIES and RTE.
This work has benefited from  the European Union's Horizon 2020 research and innovation program under the Marie Sklodowska-Curie Actions, grant agreement 813211 (POEMA) as well as from the AI Interdisciplinary Institute ANITI funding, through the French ``Investing for the Future PIA3'' program under the Grant agreement n$^{\circ}$ANR-19-PI3A-0004.

\bibliographystyle{abbrv}
\bibliography{refer}
\end{document}